\documentclass[12pt]{amsart}
\usepackage{amsmath,amssymb,amsthm,amscd,amsfonts,geometry,color}
\renewcommand{\labelenumi}{\rm(\theenumi)}

\theoremstyle{plain}
\newtheorem{thm}{Theorem}[section]
\newtheorem{prop}[thm]{Proposition}

\newtheorem{lem}[thm]{Lemma}

\newtheorem{main}{Main Theorem}

\theoremstyle{remark}
\newtheorem{remark}{Remark}

\newtheorem{prm}{Probrem}

\theoremstyle{definition}

\newcommand{\R}{\mathbb{R}}                       
\newcommand{\N}{\mathbb{N}}                       

\newcommand{\Q}{\mathbf{Q}}                       
\newcommand{\s}{\mathbf{s}}                       
\newcommand{\cs}{\mathbf{c}}                      

\newcommand{\cl}{\operatorname{cl}}               

\newcommand{\dist}{\operatorname{dist}}           

\begin{document}

\title[The space of uniformly continuous functions with the $L^p$ norm]{The space consisting of uniformly continuous functions on a metric measure space with the $L^p$ norm}
\author{Katsuhisa Koshino}
\address[Katsuhisa Koshino]{Faculty of Engineering, Kanagawa University, Yokohama, 221-8686, Japan}
\email{ft160229no@kanagawa-u.ac.jp}
\subjclass[2010]{Primary 54C35; Secondary 57N20, 57N17, 46E15, 46E30, 28A20.}
\keywords{$L^p$-space, uniformly continuous, Hilbert cube, pseudo interior, $Z$-set, strongly universal, absorbing set, absolute $F_{\sigma\delta}$-space}
\maketitle

\begin{abstract}
Let $X = (X,d,\mathcal{M},\mu)$ be a metric measure space,
 where $d$ is a metric on $X$,
 $\mathcal{M}$ is a $\sigma$-algebra of $X$,
 and $\mu$ is a measure on $\mathcal{M}$.
Suppose that $X$ is separable and locally compact,
 that $\mathcal{M}$ contains the Borel sets of $X$,
 that for each $E \in \mathcal{M}$, there exists a Borel set $B \subset X$ such that $E \subset B$ and $\mu(B \setminus E) = 0$,
 that for every non-empty open set $U \subset X$, $\mu(U) > 0$,
 that for all compact sets $K \subset X$, $\mu(K) < \infty$,
 and that $X \setminus \{x \in X \mid \{x\} \in \mathcal{M} \text{ and } \mu(\{x\}) = 0\}$ is not dense in $X$.
In this paper, we shall show that the space of real-valued uniformly continuous functions on $X$ with the $L^p$ norm, $1 \leq p < \infty$, is homeomorphic to the subspace consisting of sequences conversing to $0$ in the pseudo interior.
\end{abstract}

\section{Introduction}

Throughout this paper, we assume that spaces are Hausdorff, maps are continuous,
 but functions are not necessarily continuous,
 and $1 \leq p < \infty$.
Let $X = (X,d,\mathcal{M},\mu)$ denote a metric measure space,
 where $d$ is a metric on $X$,
 $\mathcal{M}$ is a $\sigma$-algebra of $X$,
 and $\mu$ is a measure on $\mathcal{M}$.
Denote
 $$X_0 = \{x \in X \mid \{x\} \in \mathcal{M} \text{ and } \mu(\{x\}) = 0\}.$$
A measure space $X$ is called to be \textit{Borel} provided that $\mathcal{M}$ contains the Borel sets of $X$.
We say that a Borel measure space $X$ is \textit{regular} if for each $E \in \mathcal{M}$, there is a Borel set $B \subset X$ such that $E \subset B$ and $\mu(B \setminus E) = 0$.
Notice that the Euclidean space $\R^n = (\R^n,\mathcal{M},\mu)$, where $\mathcal{M}$ is the Lebesgue measurable sets and $\mu$ is the Lebesgue measure,
 is regular Borel.
A regular Borel measure space $X$ satisfies the following stronger condition.
\begin{itemize}
 \item For any $E \in \mathcal{M}$, there exist Borel sets $B_1$ and $B_2$ in $X$ such that $B_2 \subset E \subset B_1$ and $\mu(B_1 \setminus B_2) = 0$.
\end{itemize}
Indeed, if $X$ is regular Borel,
 then for every $E \in \mathcal{M}$, we can find Borel subsets $B_1$ and $C$ of $X$ such that $E \subset B_1$, $X \setminus E \subset C$, $\mu(B_1 \setminus E) = 0$ and $\mu(C \setminus (X \setminus E)) = 0$.
Let $B_2 = X \setminus C$,
 so
 $$\mu(B_1 \setminus B_2) = \mu(B_1 \setminus E) + \mu(E \setminus B_2) = \mu(B_1 \setminus E) + \mu(C \setminus (X \setminus E)) = 0.$$

We write the integral of a real-valued $\mathcal{M}$-measurable function $f(x)$ on $E \in \mathcal{M}$ with respect to $\mu$ as $\int_E f(x) d\mu(x)$.
Set
 $$L^p(X) = \bigg\{f : X \to \R \ \bigg| \ f \text{ is $\mathcal{M}$-measurable and } \int_X |f(x)|^p d\mu(x) < \infty\bigg\}$$
 endowed with the following norm
 $$\|f\|_p = \bigg(\int_X |f(x)|^p d\mu(x)\bigg)^{1/p},$$
 where two functions that are coincident almost everywhere are identified.
Recall that a property holds \textit{almost everywhere} if it holds on $X \setminus E$ for some $E \in \mathcal{M}$ with $\mu(E) = 0$.
The function space $L^p(X)$ is a Banach space, refer to \cite[Theorem~4.8]{Br}.
It is said that $X$ is \textit{$\sigma$-finite} if it can be written as a countable union of measurable sets of finite measure.
When $X$ is a separable $\sigma$-finite regular Borel metric measure space,
 $L^p(X)$ is also separable, see \cite[Proposition~3.4.5]{Coh}.
If for any $n \in \N$, there is a pairwise disjoint family $\{E_i\}_{1 \leq i \leq n} \subset \mathcal{M}$ such that each $\mu(E_i) > 0$,
 then $L^p(X)$ is infinite-dimensional.
Hnece in the case that $X$ is infinite, $\mathcal{M}$ contains the open subsets of $X$ and for each non-empty open set $U \subset X$, $\mu(U) > 0$, the space $L^p(X)$ is infinite-dimensional.
Denote the Hilbert cube by $\Q = [-1,1]^\N$ and the pseudo interior by $\s = (-1,1)^\N$.
In the theory of infinite-dimensional topology, typical infinite-dimensional spaces, for example subspaces of $\Q$, have been detected among function spaces.
Due to the efforts of R.D.~Anderson \cite{Ande1} and M.I.~Kadec \cite{Kad}, we have the following:

\begin{thm}\label{Lp}
Let $X$ be an infinite separable $\sigma$-finite regular Borel metric measure space.
Suppose that for any non-empty open subset $U$ of $X$, the measure $\mu(U) > 0$.
Then $L^p(X)$ is an infinite-dimensional separable Banach space,
 so it is homeomorphic to $\s$.
\end{thm}

In this paper, the topological type of the subspace
 $$C_u(X) = \{f \in L^p(X) \mid f \text{ is uniformly continuous}\}\footnote{Recall that for a property $P$ of functions, a function $g \in \{f \in L^p(X) \mid f \text{ satisfies the property } P\}$ if there exists $f : X \to \R$ such that $f$ satisfies the property $P$ and $g = f$ almost everywhere.}$$
 will be studied.
When $X$ is compact,
 $C_u(X)$ is coincident with the space
 $$C(X) = \{f \in L^p(X) \mid f \text{ is continuous}\}.$$
It is known that several function spaces are homeomorphic to the following subspace of $\s$,
 $$\cs_0 = \Big\{(x(n))_{n \in \N} \in \s \ \Big| \ \lim_{n \to \infty} x(n) = 0\Big\},$$
 refer to \cite{Mil3,YZ,YSK}.
R.~Cauty \cite{Ca} proved the next theorem.

\begin{thm}
Let $[0,1] = ([0,1],d,\mathcal{M},\mu)$ be the closed unit interval,
 where $d$ is the usual metric,
 $\mathcal{M}$ is the Lebesgue measurable sets,
 and $\mu$ is the Lebesgue measure.
Then $C([0,1])$ is homeomorphic to $\cs_0$.
\end{thm}

More generally, we shall show the following:

\begin{main}
Let $X$ be a separable locally compact regular Borel metric measure space.
Suppose that for every non-empty open set $U \subset X$, $\mu(U) > 0$,
 that for each compact set $K \subset X$, $\mu(K) < \infty$,
 and that $X \setminus X_0$ is not dense in $X$.
Then $C_u(X)$ is homeomorphic to $\cs_0$.
\end{main}

\section{Preliminaries}

For each point $x \in X$ and each positive number $\delta > 0$, put the open ball $B(x,\delta) = \{y \in X \mid d(x,y) < \delta\}$.
Given subsets $A, B \subset L^p(X)$, we denote their distance by $\dist(A,B) = \inf_{f \in A, g \in B} \|f - g\|_p$.
For spaces $A \subset Y$, the symbol $\cl_Y{A}$ stands for the closure of $A$ in $Y$.
Recall that for functions $f : Z \to Y$ and $g : Z \to Y$, and for an open cover $\mathcal{U}$ of $Y$, $f$ is \textit{$\mathcal{U}$-close} to $g$ provided that for each $z \in Z$, there exists an open set $U \in \mathcal{U}$ such that the doubleton $\{f(z), g(z)\} \subset U$.
We call a closed set $A$ in a space $Y$ a \textit{$Z$-set} in $Y$ if for each open cover $\mathcal{U}$ of $Y$, there exists a map $f : Y \to Y$ such that $f$ is $\mathcal{U}$-close to the identity map of $Y$ and the image $f(Y)$ misses $A$.
A \textit{$Z_\sigma$-set} is a countable union of $Z$-sets.
A map $f : Z \to Y$ is called to be a \textit{$Z$-embedding} if $f$ is an embedding and $f(Z)$ is a $Z$-set in $Y$.
Given a class $\mathfrak{C}$ of spaces, we say that $Y$ is \textit{strongly $\mathfrak{C}$-universal} if the following condition is satisfied.
\begin{itemize}
 \item Let $A \in \mathfrak{C}$ and $f : A \to Y$ be a map.
 Suppose that $B$ is a closed set in $A$ and the restriction $f|_B$ is a $Z$-embedding.
 Then for each open cover $\mathcal{U}$ of $Y$, there is a $Z$-embedding $g : A \to Y$ such that $g$ is $\mathcal{U}$-close to $f$ and $g|_B = f|_B$.
\end{itemize}
For spaces $Y \subset M$, $Y$ is \textit{homotopy dense} in $M$ if $M$ admits a homotopy $h : M \times [0,1] \to M$ such that $h(M \times (0,1]) \subset Y$ and $h(y,0) = y$ for every $y \in M$.
For a class $\mathfrak{C}$, let $\mathfrak{C}_\sigma$ be the class of spaces written as a countable union of closed subspaces that belong to $\mathfrak{C}$.
A space $Y$ is said to be a \textit{$\mathfrak{C}$-absorbing set} in $M$ provided that it satisfies the following conditions.
\begin{enumerate}
 \item $Y \in \mathfrak{C}_\sigma$ and is homotopy dense in $M$.
 \item $Y$ is strongly $\mathfrak{C}$-universal.
 \item $Y$ is contained in a $Z_\sigma$-set in $M$.
\end{enumerate}
Let $\mathfrak{M}_2$ be the class of \textit{absolute $F_{\sigma\delta}$-spaces},
 that is, $Y \in \mathfrak{M}_2$ if $Y$ is metrizable and is an $F_{\sigma\delta}$-set in any metrizable space $M$ containing $Y$ as a subspace.
The space $\cs_0$ is an $\mathfrak{M}_2$-absorbing set in $\s$.
According to Theorem~3.1 of \cite{BeMo}, we can establish the following:

\begin{thm}\label{abs.}
Let $Y$ and $Z$ be an $\mathfrak{M}_2$-absorbing set in $\s$.
Then $Y$ and $Z$ are homeomorphic.
\end{thm}

\section{The Borel complexity of $C_u(X)$ in $L^p(X)$}

In this section, we will show that $C_u(X) \in \mathfrak{M}_2$.
The following proposition is of use, refer to Theorem~4.9 of \cite{Br}.

\begin{prop}\label{conv.}
Let $f, f_k \in L^p(X)$, $k \in \N$.
If $\|f - f_k\|_p \to 0$,
 then there exists a subsequence $\{f_{k(n)}\}$ such that $f_{k(n)} \to f$ almost everywhere.
\end{prop}

For all positive numbers $\epsilon, \delta > 0$, let
 $$A(\epsilon,\delta) = \{f \in L^p(X) \mid \text{for almost every } x, y \in X, \text{ if } d(x,y) < \delta,
 \text{ then } |f(x) - f(y)| \leq \epsilon\}.$$

\begin{lem}\label{cld.}
For any $\epsilon, \delta > 0$, the subset $A(\epsilon,\delta)$ is closed in $L^p(X)$.
\end{lem}

\begin{proof}
To prove that $A(\epsilon,\delta)$ is closed in $L^p(X)$, fix any sequence $\{f_k\}$ in $A(\epsilon,\delta)$ convergeing to $f \in L^p(X)$.
We need only to show that $f \in A(\epsilon,\delta)$,
 that is, for almost every $x, y \in X$, if $d(x,y) < \delta$,
 then $|f(x) - f(y)| \leq \epsilon$.
Since $\|f - f_k\|_p \to 0$,
 we can replace $\{f_k\}$ with a subsequence so that $f_k \to f$ almost everywhere by Proposition~\ref{conv.}.
Then there exists $E_0 \subset X$ such that $\mu(E_0) = 0$ and $f_k(x) \to f(x)$ for each $x \in X \setminus E_0$.
On the other hand, because $f_k \in A(\epsilon,\delta)$ for each $k \in \N$, we can find $E_k \subset X$ with $\mu(E_k) = 0$ so that for any $x, y \in X \setminus E_k$, if $d(x,y) < \delta$,
 then $|f_k(x) - f_k(y)| \leq \epsilon$.
Let $E = \bigcup_{k \in \N \cup \{0\}} E_k$ and take any $x, y \in X \setminus E$ with $d(x,y) < \delta$.
Note that $\mu(E) = 0$.
Then $|f_k(x) - f_k(y)| \leq \epsilon$ for every $k \in \N$, $f_k(x) \to f(x)$ and $f_k(y) \to f(y)$,
 which implies that $|f(x) - f(y)| \leq \epsilon$.
We conclude that $f \in A(\epsilon,\delta)$.
\end{proof}

Note that a space $Y \in \mathfrak{M}_2$ if and only if $Y$ can be embedded into a completely metrizable space as an $F_{\sigma\delta}$-set,
 see \cite[Theorem~9.6]{AN}.
We prove the following:

\begin{prop}\label{M2}
Suppose that for each $E \in \mathcal{M}$ with $\mu(E) = 0$, $X \setminus E$ is dense in $X$.
Then $C_u(X)$ is an $F_{\sigma\delta}$-set in $L^p(X)$,
 and hence $C_u(X) \in \mathfrak{M}_2$.
\end{prop}

\begin{proof}
We shall show that $C_u(X) = \bigcap_{n \in \N} \bigcup_{m \in \N} A(1/n,1/m)$.
Then it follows from Lemma~\ref{cld.} that $C_u(X)$ is $F_{\sigma\delta}$ in $L^p(X)$.
Clearly, $C_u(X) \subset \bigcap_{n \in \N} \bigcup_{m \in \N} A(1/n,1/m)$.
To prove that $C_u(X) \supset \bigcap_{n \in \N} \bigcup_{m \in \N} A(1/n,1/m)$, fix any function $f \in \bigcap_{n \in \N} \bigcup_{m \in \N} A(1/n,1/m)$.
For each $n \in \N$, there exist $m_n \in \N$ and $E_n \subset X$ with $\mu(E_n) = 0$ such that if $x, y \in X \setminus E_n$ and $d(x,y) < 1/m_n$,
 then $|f(x) - f(y)| \leq 1/n$.
Let $E = \bigcup_{n \in \N} E_n$,
 so the measure $\mu(E) = 0$.
Then for any $n \in \N$ and $x, y \in X \setminus E$ with $d(x,y) < 1/m_n$,
 we have $|f(x) - f(y)| \leq 1/n$,
 which implies that $f|_{X \setminus E}$ is uniformly continuous.
Since $X \setminus E$ is dense in $X$,
 the restriction $f|_{X \setminus E}$ can be extended over $X$ as a uniformly continuous map,
 that is coincident with $f$ almost everywhere.
Therefore $f \in C_u(X)$.
The proof is complete.
\end{proof}

\section{The $Z_\sigma$-set property of $C_u(X)$ in $L^p(X)$}

In this section, it is shown that $C_u(X)$ is contained in some $Z_\sigma$-set in $L^p(X)$.
The next theorem is important on convergence of sequences in $L^p(X)$,
 refer to Theorem~4.2 of \cite{Br}.\footnote{This is valid for any $p \in [1,\infty)$.}

\begin{thm}[The Dominated Convergence Theorem]\label{dom.conv.}
Let $f, f_k \in L^1(X)$, $k \in \N$.
Suppose that $f_k(x) \to f(x)$ for almost every $x \in X$,
 and that there is a function $g \in L^p(X)$ such that for any $k \in \N$, $|f_k(x)| \leq g(x)$ for almost every $x \in X$.
Then $f, f_k \in L^p(X)$, $k \in \N$,
 and $\|f - f_k\|_p \to 0$.
\end{thm}

The following technical lemma will be very useful for detecting $Z$-sets in $L^p(X)$.

\begin{lem}\label{dig}
Let $Y$ be a paracompact space, $\phi : Y \to L^p(X)$ be a map,
 and $a \in X_0$ such that for all $\lambda > 0$, $B(a,\lambda) \in \mathcal{M}$.
Then for every map $\epsilon : Y \to (0,1)$, there exist maps $\psi : Y \to L^p(X)$ and $\delta : Y \to (0,1)$ such that for each $y \in Y$,
\begin{enumerate}
 \renewcommand{\labelenumi}{(\roman{enumi})}
 \item $\|\phi(y) - \psi(y)\|_p \leq \epsilon(y)$,
 \item $\psi(y)(B(a,\delta(y))) = \{0\}$.
\end{enumerate}
\end{lem}

\begin{proof}
For each $f \in L^p(X)$ and each $A \in \mathcal{M}$, define a function $f_A \in L^p(X)$ by
 $$f_A(x) = \left\{
 \begin{array}{ll}
 f(x) &\text{if } x \in A,\\
 0 &\text{if } x \in X \setminus A.
 \end{array}
 \right.$$
Remark that $\|f_A\|_p = (\int_A |f(x)|^p d\mu(x))^{1/p}$.
Given any $y\in Y$, put
 $$\xi(y) = \sup\{0 < \eta \leq 1 \mid \|\phi(y)_{B(a,\eta)}\|_p < \epsilon(y)\}.$$
According to Theorem~\ref{dom.conv.}, $\xi(y) > 0$ for all $y \in Y$ because $\mu(\{a\}) = 0$.
Then the function $\xi : Y \to (0,1]$ is lower semi-continuous.
Indeed, fix any $y \in Y$ and $\eta \in (0,\xi(y))$.
By the definition, there is $\lambda \in (0,\epsilon(y))$ such that $\|\phi(y)_{B(a,\xi(y)-\eta)}\|_p < \epsilon(y) - \lambda$.
Due to the continuity of $\phi$ and $\epsilon$, we can find a neighborhood $U$ of $y$ such that for every $y' \in U$, $\|\phi(y) - \phi(y')\|_p < \lambda/2$ and $|\epsilon(y) - \epsilon(y')| < \lambda/2$.
Then
\begin{align*}
 \|\phi(y')_{B(a,\xi(y)-\eta)}\|_p &\leq \|\phi(y')_{B(a,\xi(y)-\eta)} - \phi(y)_{B(a,\xi(y)-\eta)}\|_p + \|\phi(y)_{B(a,\xi(y)-\eta)}\|_p\\
 &\leq \|\phi(y') - \phi(y)\|_p + \|\phi(y)_{B(a,\xi(y)-\eta)}\|_p\\
 &< \epsilon(y) - \lambda/2 < \epsilon(y').
\end{align*}
Therefore $\xi(y') \geq \xi(y) - \eta$,
 which means that $\xi$ is a lower semi-continuous function.
Since $Y$ is paracompact,
 there is a map $\delta : Y \to (0,1)$ such that $\delta(y) < \xi(y)/2$ for each $y \in Y$ by virtue of Theorem~2.7.6 of \cite{Sakaik11}.
Then the desired map $\psi : Y \to L^p(X)$ can be defined as follows:
 $$\psi(y)(x) = \left\{
  \begin{array}{ll}
  0 &\text{if } x \in B(a,\delta(y)),\\
  ((d(a,x) - \delta(y))/\delta(y))\phi(y)(x) &\text{if } x \in B(a,2\delta(y)) \setminus B(a,\delta(y)),\\
  \phi(y)(x) &\text{if } x \in X \setminus B(a,2\delta(y)).
  \end{array}
 \right.$$
Condition (ii) follows from the definition immediately.
Let us note that for each $x \in B(a,2\delta(y)) \setminus B(a,\delta(y))$,
\begin{align*}
 |\phi(y)(x) - \psi(y)(x)| &= |\phi(y)(x) - ((d(a,x) - \delta(y))/\delta(y))\phi(y)(x)|\\
 &= ((2\delta(y) - d(a,x))/\delta(y))|\phi(y)(x)| \leq |\phi(y)(x)|.
\end{align*}
Since $2\delta(y) < \xi(y)$,
 we get
 $$\|\phi(y) - \psi(y)\|_p \leq \|\phi(y)_{B(a,2\delta(y))}\|_p \leq \|\phi(y)_{B(a,\xi(y))}\|_p \leq \epsilon(y),$$
 and hence condition (i) holds.

It remains to verify the continuity of $\psi$.
Take any $y \in Y$ and $\lambda > 0$.
Since $\phi$ and $\delta$ are continuous,
 we can choose a neighborhood $U$ of $y$ so that for each $y' \in U$,
\begin{enumerate}
 \renewcommand{\labelenumi}{(\alph{enumi})}
 \item $\|\phi(y) - \phi(y')\|_p < \lambda/8$,
 \item $|\delta(y) - \delta(y')| < \delta(y)$/2,
 \item $|\delta(y) - \delta(y')|\|\phi(y)\|_p < \lambda\delta(y)/8$,
 \item $|1/\delta(y) - 1/\delta(y')|\|\phi(y)\|_p < \lambda/8$.
\end{enumerate}
We shall prove that $\|\psi(y) - \psi(y')\|_p < \lambda$ only in the case that $\delta(y) \leq \delta(y')$ because it can be shown similarly in the other case.
Note that $\delta(y') < 2\delta(y)$ by condition (b).
Obviously, $\|\psi(y)_{B(a,\delta(y))} - \psi(y')_{B(a,\delta(y))}\|_p = 0$.
Due to condition (c), we have
\begin{align*}
 &\|\psi(y)_{B(a,\delta(y')) \setminus B(a,\delta(y))} - \psi(y')_{B(a,\delta(y')) \setminus B(a,\delta(y))}\|_p\\
 & \ \ \ \ \ \ \ \ = \bigg(\int_{B(a,\delta(y')) \setminus B(a,\delta(y))} |\psi(y)(x) - \psi(y')(x)|^p d\mu(x)\bigg)^{1/p}\\
 & \ \ \ \ \ \ \ \ = \bigg(\int_{B(a,\delta(y')) \setminus B(a,\delta(y))} |((d(a,x) - \delta(y))/\delta(y))\phi(y)(x)|^p d\mu(x)\bigg)^{1/p}\\
 & \ \ \ \ \ \ \ \ \leq ((\delta(y') - \delta(y))/\delta(y))\bigg(\int_{B(a,\delta(y')) \setminus B(a,\delta(y))} |\phi(y)(x)|^p d\mu(x)\bigg)^{1/p}\\
 & \ \ \ \ \ \ \ \ \leq ((\delta(y') - \delta(y))/\delta(y))\bigg(\int_X |\phi(y)(x)|^p d\mu(x)\bigg)^{1/p} = ((\delta(y') - \delta(y))\|\phi(y)\|_p/\delta(y) < \lambda/8.
\end{align*}
By conditions (a) and (d),
\begin{align*}
 &\|\psi(y)_{B(a,2\delta(y)) \setminus B(a,\delta(y'))} - \psi(y')_{B(a,2\delta(y)) \setminus B(a,\delta(y'))}\|_p\\
 & \ \ \ \ \ \ \ \ = \bigg(\int_{B(a,2\delta(y)) \setminus B(a,\delta(y'))} |\psi(y)(x) - \psi(y')(x)|^p d\mu(x)\bigg)^{1/p}\\
 & \ \ \ \ \ \ \ \ = \bigg(\int_{B(a,2\delta(y)) \setminus B(a,\delta(y'))} |((d(a,x) - \delta(y))/\delta(y))\phi(y)(x)\\
 & \ \ \ \ \ \ \ \ \ \ \ \ \ \ \ \ \ \ \ \ \ \ \ \ \ \ \ \ \ \ \ \ \ \ \ \ \ \ \ \ \ \ \ \ \ \ \ \ - ((d(a,x) - \delta(y'))/\delta(y'))\phi(y')(x)|^p d\mu(x)\bigg)^{1/p}\\
 & \ \ \ \ \ \ \ \ = \bigg(\int_{B(a,2\delta(y)) \setminus B(a,\delta(y'))} |((d(a,x) - \delta(y))/\delta(y) - (d(a,x) - \delta(y'))/\delta(y'))\phi(y)(x)\\
 & \ \ \ \ \ \ \ \ \ \ \ \ \ \ \ \ \ \ \ \ \ \ \ \ \ \ \ \ \ \ \ \ \ \ \ \ \ \ \ \ \ \ \ \ \ \ \ \ + ((d(a,x) - \delta(y'))/\delta(y'))(\phi(y)(x) - \phi(y')(x))|^p d\mu(x)\bigg)^{1/p}\\
 & \ \ \ \ \ \ \ \ \leq (1/\delta(y) - 1/\delta(y'))\bigg(\int_{B(a,2\delta(y)) \setminus B(a,\delta(y'))} |\phi(y)(x)|^p d\mu(x)\bigg)^{1/p}\\
 & \ \ \ \ \ \ \ \ \ \ \ \ \ \ \ \ \ \ \ \ \ \ \ \ \ \ \ \ \ \ \ \ \ \ \ \ \ \ \ \ \ \ \ \ \ \ \ \ + \bigg(\int_{B(a,2\delta(y)) \setminus B(a,\delta(y'))} |\phi(y)(x) - \phi(y')(x)|^p d\mu(x)\bigg)^{1/p}\\
 & \ \ \ \ \ \ \ \ \leq (1/\delta(y) - 1/\delta(y'))\bigg(\int_X |\phi(y)(x)|^p d\mu(x)\bigg)^{1/p} + \bigg(\int_X |\phi(y)(x) - \phi(y')(x)|^p d\mu(x)\bigg)^{1/p}\\
 & \ \ \ \ \ \ \ \ = (1/\delta(y) - 1/\delta(y'))\|\phi(y)\|_p + \|\phi(y) - \phi(y')\|_p < \lambda/4.
\end{align*}
Using conditions (a) and (c), we get
\begin{align*}
 &\|\psi(y)_{B(a,2\delta(y')) \setminus B(a,2\delta(y))} - \psi(y')_{B(a,2\delta(y')) \setminus B(a,2\delta(y))}\|_p\\
 & \ \ \ \ \ \ \ \ = \bigg(\int_{B(a,2\delta(y')) \setminus B(a,2\delta(y))} |\psi(y)(x) - \psi(y')(x)|^p d\mu(x)\bigg)^{1/p}\\
 & \ \ \ \ \ \ \ \ = \bigg(\int_{B(a,2\delta(y')) \setminus B(a,2\delta(y))} |\phi(y)(x) - ((d(a,x) - \delta(y'))/\delta(y'))\phi(y')(x)|^p d\mu(x)\bigg)^{1/p}\\
 & \ \ \ \ \ \ \ \ = \bigg(\int_{B(a,2\delta(y')) \setminus B(a,2\delta(y))} |(1 - (d(a,x) - \delta(y'))/\delta(y'))\phi(y)(x)\\
 & \ \ \ \ \ \ \ \ \ \ \ \ \ \ \ \ \ \ \ \ \ \ \ \ \ \ \ \ \ \ \ \ \ \ \ \ \ \ \ \ \ \ \ \ \ \ \ \ + ((d(a,x) - \delta(y'))/\delta(y'))(\phi(y)(x) - \phi(y')(x))|^p d\mu(x)\bigg)^{1/p}\\
 & \ \ \ \ \ \ \ \ \leq 2((\delta(y') - \delta(y))/\delta(y))\bigg(\int_{B(a,2\delta(y')) \setminus B(a,2\delta(y))} |\phi(y)(x)|^p d\mu(x)\bigg)^{1/p}\\
 & \ \ \ \ \ \ \ \ \ \ \ \ \ \ \ \ \ \ \ \ \ \ \ \ \ \ \ \ \ \ \ \ \ \ \ \ \ \ \ \ \ \ \ \ \ \ \ \ + \bigg(\int_{B(a,2\delta(y')) \setminus B(a,2\delta(y))} |\phi(y)(x) - \phi(y')(x)|^p d\mu(x)\bigg)^{1/p}\\
 & \ \ \ \ \ \ \ \ \leq 2((\delta(y') - \delta(y))/\delta(y))\bigg(\int_X |\phi(y)(x)|^p d\mu(x)\bigg)^{1/p} + \bigg(\int_X |\phi(y)(x) - \phi(y')(x)|^p d\mu(x)\bigg)^{1/p}\\
 & \ \ \ \ \ \ \ \ = 2((\delta(y') - \delta(y))/\delta(y))\|\phi(y)\|_p + \|\phi(y) - \phi(y')\|_p < 3\lambda/8.
\end{align*}
It follows from condition (a) that
\begin{align*}
 \|\psi(y)_{X \setminus B(a,2\delta(y'))} - \psi(y')_{X \setminus B(a,2\delta(y'))}\|_p &= \bigg(\int_{X \setminus B(a,2\delta(y'))} |\psi(y)(x) - \psi(y')(x)|^p d\mu(x)\bigg)^{1/p}\\
 &= \bigg(\int_{X \setminus B(a,2\delta(y'))} |\phi(y)(x) - \phi(y')(x)|^p d\mu(x)\bigg)^{1/p}\\
 &\leq \bigg(\int_X |\phi(y)(x) - \phi(y')(x)|^p d\mu(x)\bigg)^{1/p}\\
 &= \|\phi(y) - \phi(y')\|_p < \lambda/8.
\end{align*}
Therefore we have
\begin{align*}
 \|\psi(y) - \psi(y')\|_p &\leq \|\psi(y)_{B(a,\delta(y))} - \psi(y')_{B(a,\delta(y))}\|_p\\
 & \ \ \ \ \ \ \ \ + \|\psi(y)_{B(a,\delta(y')) \setminus B(a,\delta(y))} - \psi(y')_{B(a,\delta(y')) \setminus B(a,\delta(y))}\|_p\\
 & \ \ \ \ \ \ \ \ + \|\psi(y)_{B(a,2\delta(y)) \setminus B(a,\delta(y'))} - \psi(y')_{B(a,2\delta(y)) \setminus B(a,\delta(y'))}\|_p\\
 & \ \ \ \ \ \ \ \ + \|\psi(y)_{B(a,2\delta(y')) \setminus B(a,2\delta(y))} - \psi(y')_{B(a,2\delta(y')) \setminus B(a,2\delta(y))}\|_p\\
 & \ \ \ \ \ \ \ \ + \|\psi(y)_{X \setminus B(a,2\delta(y'))} - \psi(y')_{X \setminus B(a,2\delta(y'))}\|_p\\
 &< 7\lambda/8 < \lambda.
\end{align*}
Consequently, $\psi$ is continuous.
Thus the proof is completed.
\end{proof}

\begin{remark}\label{target}
In the above lemma, for each $y \in Y$, when $\phi(y)$ is corresponding to a function almost everywhere,
 that is uniformly continuous and bounded on $B(a,2\delta(y))$,
 we have $\psi(y) \in C_u(X)$.
\end{remark}

We show the following:

\begin{lem}\label{raise}
Let $a \in X_0$ such that for each $\lambda > 0$, $B(a,\lambda) \in \mathcal{M}$ and $\mu(B(a,\lambda)) > 0$,
 and for some $\lambda' > 0$, $\mu(B(a,\lambda')) < \infty$.
Suppose that $A \subset L^p(X)$ and $\xi : A \to (0,\infty)$ is a function such that for every $f \in A$, $f(x) = 0$ for almost every $x \in B(a,\xi(f))$,
 and that $B$ is a $Z$-set in $L^p(X)$.
If the union $A \cup B$ is a closed set in $L^p(X)$,
 then it is a $Z$-set.
\end{lem}

\begin{proof}
Let $\epsilon : L^p(X) \to (0,1)$ be a map.
We shall construct a map $\phi : L^p(X) \to L^p(X)$ so that $\phi(L^p(X)) \cap (A \cup B) = \emptyset$ and $\|\phi(f) - f\|_p < \epsilon(f)$ for every $f \in L^p(X)$.
Since $B$ is a $Z$-set,
 there is a map $\psi_1 : L^p(X) \to L^p(X) \setminus B$ such that $\|\psi_1(f) - f\|_p < \epsilon(f)/3$ for each $f \in L^p(X)$.
Using Lemma~\ref{dig}, we can obtain maps $\psi_2 : L^p(X) \to L^p(X)$ and $\delta : L^p(X) \to (0,1)$ such that for each $f \in L^p(X)$,
\begin{enumerate}
 \renewcommand{\labelenumi}{(\roman{enumi})}
 \item $\|\psi_1(f) - \psi_2(f)\|_p \leq \min\{\epsilon(f),\dist(\{\psi_1(f)\},B)\}/3$,
 \item $\psi_2(f)(B(a,\delta(f))) = \{0\}$.
\end{enumerate}
Since $\mu(B(a,\lambda')) < \infty$ for some $\lambda' > 0$,
 $$\lim\limits_{k \to \infty} \mu(B(a,\lambda'/k)) = \mu\bigg(\bigcap_{k \in \N} B(a,\lambda'/k)\bigg) = \mu(\{a\}) = 0.$$
 So we can take $\lambda > 0$ so that $\mu(B(a,\lambda)) \leq 1$.
Letting $\psi_3 : L^p(X) \to L^p(X)$ be a map such that
 $$\psi_3(f)(x) = \left\{
 \begin{array}{ll}
 \min\{\epsilon(f),\dist(\{\psi_1(f)\},B)\}/3 &\text{if } x \in B(a,\lambda),\\
 0 &\text{if } x \in X \setminus B(a,\lambda),
 \end{array}
 \right.$$
 we can defined the desired map $\phi : L^p(X) \to L^p(X)$ by $\phi(f) = \psi_2(f) + \psi_3(f)$.
Since $\psi_2$ and $\psi_3$ are continuous,
 so is $\phi$.
It is easy to see that $\phi(f) \notin A$ for any $f \in L^p(X)$.
Observe that by condition (i),
\begin{align*}
 \|\phi(f) - \psi_1(f)\|_p &= \|\psi_2(f) + \psi_3(f) - \psi_1(f)\|_p \leq \|\psi_2(f) - \psi_1(f)\|_p + \|\psi_3(f)\|_p\\
 &\leq 2\min\{\epsilon(f),\dist(\{\psi_1(f)\},B)\}/3 < \dist(\{\psi_1(f)\},B),
\end{align*}
 which implies that $\phi(f) \notin B$.
Moreover,
\begin{align*}
 \|\phi(f) - f\|_p &\leq \|\phi(f) - \psi_1(f)\|_p + \|\psi_1(f) - f\|_p\\
 &< 2\min\{\epsilon(f),\dist(\{\psi_1(f)\},B)\}/3 + \epsilon(f)/3 \leq \epsilon(f).
\end{align*}
The proof is completed.
\end{proof}

Now we will prove that there exists a $Z_\sigma$-set in $L^p(X)$ which contains $C_u(X)$.
Set
 $$C_a(X) = \{f \in L^p(X) \mid f|_{X \setminus E} \text{ is continuous for some } E \subset X \text{ with } \mu(E) = 0\}.$$
It is obvious that $C_u(X) \subset C_a(X)$.
We show the following proposition.

\begin{prop}\label{Zsigma}
Let $X$ be separable.
Suppose that for all points $x \in X_0$, $B(x,\lambda) \in \mathcal{M}$ and $\mu(B(x,\lambda)) > 0$ for each $\lambda > 0$,
 and $B(x,\lambda'(x)) < \infty$ for some $\lambda'(x) > 0$,
 and that $X \setminus X_0$ is not dense in $X$.
Then $C_a(X)$ is contained in some $Z_\sigma$-set in $L^p(X)$,
 and hence so is $C_u(X)$.
\end{prop}

\begin{proof}
Notice that $X \setminus \cl_X{(X \setminus X_0)} \neq \emptyset$ and there is a countable open basis $\mathcal{U}$ of $X \setminus \cl_X{(X \setminus X_0)}$.
We may assume that $\emptyset \notin \mathcal{U}$.
For each $n \in \N$ and each $U \in \mathcal{U}$, let
 $$Z(n,U) = \{f \in L^p(X) \mid |f(x)| \geq 1/n \text{ for almost every } x \in U\}.$$
Then $Z(n,U)$ is closed in $L^p(X)$.
Indeed, for every sequence $\{f_k\} \subset Z(n,U)$ that converges to $f \in L^p(X)$, by Proposition~\ref{conv.}, replacing $\{f_k\}$ with a subsequence, we have that $f_k \to f$ almost everywhere.
For almost every $x \in U$, $|f_k(x)| \geq 1/n$ and $f_k(x) \to f(x)$,
 which implies that $|f(x)| \geq 1/n$.
Thus $Z(n,U)$ is closed.
Fix any $a \in U$.
According to Lemma~\ref{dig}, for each map $\epsilon : L^p(X) \to (0,1)$, we can choose maps $\phi : L^p(X) \to L^p(X)$ and $\delta : L^p(X) \to (0,1)$ satisfying the following:
\begin{enumerate}
 \renewcommand{\labelenumi}{(\roman{enumi})}
 \item $\|\phi(f) - f\|_p < \epsilon(f)$, and
 \item $\phi(f)(B(a,\delta(f))) = \{0\}$ for any $f \in L^p(X)$.
\end{enumerate}
Recall that $\mu(B(a,\delta(f))) > 0$.
As is easily observed,
 $\phi(L^p(X)) \cap Z(n,U) = \emptyset$.
Hence $Z(n,U)$ is a $Z$-set in $L^p(X)$.

Let $Z = C_a(X) \setminus \bigcup_{n \in \N} \bigcup_{U \in \mathcal{U}} Z(n,U)$.
We shall show that $\cl_{L^p(X)}{Z}$ is a $Z$-set in $L^p(X)$.
Take any $a \in X \setminus \cl_X{(X \setminus X_0)}$ and $\delta > 0$ such that $B(a,\delta) \subset X \setminus \cl_X{(X \setminus X_0)}$.
For each $f \in \cl_{L^p(X)}{Z}$, we prove that $f(x) = 0$ for almost every $x \in B(a,\delta)$.
There exists $\{f_k\} \subset Z$ such that $\|f_k - f\|_p \to 0$.
By Proposition~\ref{conv.}, replacing $\{f_k\}$ with a subsequence, we can choose $E_0 \subset X$ with $\mu(E_0) = 0$ so that $f_k(x) \to f(x)$ for any $x \in B(a,\delta) \setminus E_0$.
Since each $f_k \in C_a(X)$,
 there is $E_k \subset X$ such that $\mu(E_k) = 0$ and $f_k|_{X \setminus E_k}$ is continuous.
Put $E = \bigcup_{k \in \N \cup \{0\}} E_k$,
 so $\mu(E) = 0$.
Let any $x \in B(a,\delta) \setminus E$.
For all $n \in \N$ and $U \in \mathcal{U}$, there is a point $x_{(n,U)} \in U \setminus E$ such that $|f_k(x_{(n,U)})| < 1/n$ because $f_k \notin Z(n,U)$.
Due to the continuity of $f_k|_{X \setminus E}$, we have $|f_k(x)| \leq 1/n$,
 which means that $|f(x)| \leq 1/n$.
Therefore $f(x) = 0$ for almost every $x \in B(a,\delta)$.
Consequently, $Z$ is a $Z$-set in $L^p(X)$ by Lemma~\ref{raise},
 so $C_a(X)$ is contained in the $Z_\sigma$-set $\cl_{L^p(X)}{Z} \cup \bigcup_{n \in \N} \bigcup_{U \in \mathcal{U}} Z(n,U)$.
\end{proof}

\section{The strong $\mathfrak{M}_2$-universality of $C_u(X)$}

This section is devoted to proving that $C_u(X)$ is strongly $\mathfrak{M}_2$-universal.
Indeed, we will show the stronger result in Proposition~\ref{str.univ.pair}.
Given any pair of spaces $(M,Y)$, which means that $Y \subset M$,
 and any pair of classes $(\mathfrak{A},\mathfrak{C})$, we write $(M,Y) \in (\mathfrak{A},\mathfrak{C})$ if $M \in \mathfrak{A}$ and $Y \in \mathfrak{C}$.
A pair $(M,Y)$ is called to be \textit{strongly $(\mathfrak{A},\mathfrak{C})$-universal} if the following condition holds.
\begin{itemize}
 \item Let $(A,D) \in (\mathfrak{A},\mathfrak{C})$ and $B$ be a closed subset of $A$.
 Suppose that $f : A \to M$ is a map such that $f|_B$ is a $Z$-embedding and $(f|_B)^{-1}(Y) = B \cap D$.
 Then for each open cover $\mathcal{U}$ of $M$, there exists a $Z$-embedding $g : A \to M$ such that $g$ is $\mathcal{U}$-close to $f$, $g|_B = f|_B$ and $g^{-1}(Y) = D$.
\end{itemize}
Denote the class of compact metrizable spaces by $\mathfrak{M}_0$.
Set
 $$\cs_1 = \Big\{(x(n))_{n \in \N} \in \s \ \Big| \ \lim_{n \to \infty} x(n) = 1\Big\}.$$
It is well known that the both pairs $(\s,\cs_0)$ and $(\s,\cs_1)$ are strongly $(\mathfrak{M}_0,\mathfrak{M}_2)$-universal,
 refer to \cite{Mil3}.
A strong universality of a pair implies one of a space.
By virtue of Theorems~1.7.9 and 1.3.2 of \cite{BRZ}, we can establish the following:

\begin{prop}\label{str.univ.}
Let $(\mathfrak{A},\mathfrak{C})$ be a pair of classes of metrizable spaces.
Suppose that $M$ is a space homeomorphic to $\s$ and $Y$ is a homotopy dense subspace in $M$.
If $(M,Y)$ is strongly $(\mathfrak{A},\mathfrak{C})$-universal,
 then $Y$ is strongly $\mathfrak{C}$-universal.
\end{prop}

The next lemma will be used for proving Proposition~\ref{str.univ.pair}.

\begin{lem}\label{div.}
Let $Y$ be a space and $g : Y \to \s$ be an injective map.
Suppose that for every $E \in \mathcal{M}$ with $\mu(E) = 0$, $X \setminus E$ is dense in $X$,
 and that $x_m, x_\infty \in X$, $m \in \N$, are points such that $d(x_1,x_\infty) < 1$, $\{d(x_m,x_\infty)\}$ is a strictly decreasing sequence converging to $0$ and $B(x_\infty,d(x_1,x_\infty)) \in \mathcal{M}$ with $\mu(B(x_\infty,d(x_1,x_\infty))) \leq 1$.
Then for each map $\delta : Y \to (0,1)$, there exists an injective map $\Phi : Y \to L^p(X)$ which satisfies the following conditions for every $y \in Y$.
\begin{enumerate}
 \item $\|\Phi(y)\|_p \leq \delta(y)$.
 \item $\Phi(y)(X \setminus B(x_\infty,d(x_{2k},x_\infty))) = \{0\}$ if $2^{-k} \leq \delta(y) \leq 2^{-k+1}$, $k \in \N$.
 \item $\Phi(y)(x_m) = \delta(y)$ for all $m \in \{2j + 1, \infty \mid j > k\}$ if $2^{-k} \leq \delta(y) \leq 2^{-k+1}$, $k \in \N$.
 \item $\Phi(y)$ is continuous on $X \setminus \{x_\infty\}$.
 \item $y \in g^{-1}(\cs_1)$ if and only if $\Phi(y) \in C(X)$.
\end{enumerate}
\end{lem}

\begin{proof}
For each $k \in \N$, setting
 $$Y_k = \{y \in Y \mid 2^{-k} \leq \delta(y) \leq 2^{-k+1}\},$$
 we have that $Y = \bigcup_{k \in \N} Y_k$.
Define a map $f_i^k : Y_k \to [0,1]$ as follows:
 $$f_i^k(y) = \left\{
 \begin{array}{ll}
  0 &\text{if } i = 1,\\
  \delta(y)(1-\phi_k(y)) &\text{if } i = 2,\\
  \delta(y)(1-\phi_k(y))g(y)(1) & \text{if } i = 3,\\
  \delta(y) &\text{if } i = 2j, j \geq 2,\\
  \delta(y)((1-\phi_k(y))g(y)((i-1)/2) + \phi_k(y)g(y)((i-3)/2)) &\text{if } i = 2j+1, j \geq 2,
 \end{array}
 \right.$$
 where $\phi_k(y) = 2 - 2^k\delta(y)$.
For each $m \in \N$, let
 $$S_m = \{x \in X \mid r_m \leq d(x,x_\infty) \leq r_{m-1}\},$$
 where $r_0 = 1$ and $r_m = d(x_m,x_\infty)$,
 and let $\psi_m : S_m \to [0,1]$ be a map defined by
 $$\psi_m(x) = (d(x,x_\infty) - r_m)/(r_{m-1} - r_m).$$
We define a map $\Phi_k : Y_k \to L^p(X)$, $k \in \N$, as follows:
 $$\Phi_k(y)(x) = \left\{
 \begin{array}{ll}
  \delta(y) &\text{if } x = x_\infty,\\
  \psi_{2k+i}(x)f_i^k(y) + (1-\psi_{2k+i}(x))f_{i+1}^k(y) &\text{if } x \in S_{2k+i},\\
  0 &\text{if } d(x,x_\infty) \geq r_{2k}.
 \end{array}
 \right.$$
Verify that $\Phi_k(y) = \Phi_{k+1}(y)$ for all $y \in Y_k \cap Y_{k+1}$.
Indeed, by the definition, $\Phi_k(y)(x_\infty) = \delta(y) = \Phi_{k+1}(y)(x_\infty)$,
 and $\Phi_k(y)(x) = 0 = \Phi_{k + 1}(y)(x)$ for every $x \in X$ with $d(x,x_\infty) \geq r_{2k}$.
We get $\phi_k(y) = 1$ and $\phi_{k+1}(y) = 0$ because $\delta(y) = 2^{-k}$,
 and hence $f_1^k(y) = f_2^k(y) = f_3^k(y) = 0$.
Therefore for each $x \in S_{2k + 1}$,
 $$\Phi_k(y)(x) = \psi_{2k + 1}(x)f_1^k(y)+(1-\psi_{2k+1}(x))f_2^k(y) = 0 = \Phi_{k + 1}(y)(x),$$
 and for each $x \in S_{2k + 2}$,
 $$\Phi_k(y)(x) = \psi_{2k+2}(x)f_2^k(y)+(1-\psi_{2k+2}(x))f_3^k(y) = 0 = \Phi_{k + 1}(y)(x).$$
Moreover, $f_3^k(y) = 0 = f_1^{k + 1}(y)$, $f_{2j+3}^k(y) = \delta(y)g(y)(j) = f_{2j + 1}^{k + 1}(y)$ and $f_{2j + 2}^k(y) = \delta(y) = f_{2j}^{k + 1}(y)$ for any $j \geq 1$,
 that is, $f_{i + 2}^k(y) = f_i^{k + 1}(y)$ for any $i \geq 1$.
It follows that for each $x \in S_{2k + i + 2}$, $i \geq 1$,
\begin{align*}
 \Phi_k(y)(x) &= \psi_{2k + i + 2}(x)f_{i + 2}^k(y) + (1 - \psi_{2k + i + 2}(x))f_{i + 3}^k(y)\\
 &= \psi_{2(k + 1) + i}(x)f_i^{k + 1}(y) + (1 - \psi_{2(k + 1) + i}(x))f_{i + 1}^{k + 1}(y) = \Phi_{k + 1}(y)(x).
\end{align*}
As a consequence, $\Phi_k(y) = \Phi_{k+1}(y)$.

Now define the desired map $\Phi : Y \to L^p(X)$ by $\Phi(y) = \Phi_k(y)$ if $y \in Y_k$.
Evidently, conditions (1), (2), (3) and (4) follows from the definition of $\Phi$.
We will check condition (5).
Firstly, let us show the only if part.
Take any $y \in g^{-1}(\cs_1)$,
 where $y \in Y_k$ for some $k \in \N$,
 and let $\epsilon \in (0,\delta(y))$.
Since $g(y) \in \cs_1$,
 there exists $i_0 \in \N$ such that if $i \geq i_0$,
 then $g(y)(i) > 1 - \epsilon/\delta(y)$.
Let any $i \geq 2i_0 + 3$ and any point $x \in S_{2k + i}$.
In the case that $i$ is even,
 $f_i^k(y) = \delta(y)$.
In the case that $i$ is odd,
\begin{align*}
 f_i^k(y) &= \delta(y)((1-\phi_k(y))g(y)((i-1)/2) + \phi_k(y)g(y)((i-3)/2))\\
 &> \delta(y)((1-\phi_k(y))(1-\epsilon/\delta(y))+\phi_k(y)(1-\epsilon/\delta(y))) = \delta(y) - \epsilon.
\end{align*}
Therefore we get that
\begin{align*}
 \psi_{2k+i}(x)f_i^k(y) + (1-\psi_{2k+i}(x))f_{i+1}^k(y) &> \psi_{2k+i}(x)(\delta(y)-\epsilon) + (1-\psi_{2k+i}(x))(\delta(y)-\epsilon)\\
 &= \delta(y) - \epsilon.
\end{align*}
It follows that
\begin{align*}
 |\Phi(y)(x_\infty) - \Phi(y)(x)| &= |\delta(y) - (\psi_{2k+i}(x)f_i^k(y) + (1-\psi_{2k+i}(x))f_{i+1}^k(y))|\\
 &< \delta(y) - (\delta(y)-\epsilon) = \epsilon,
\end{align*}
 which means that the function $\Phi(y)$ is continuous at $x_\infty$.
Moreover, $\Phi(y)$ is continuous on $X \setminus \{x_\infty\}$ due to (4),
 so $\Phi(y) \in C(X)$.

Next, to prove the if part, fix any $y \in Y$ such that $\Phi(y) \in C(X)$.
Then $y \in Y_k$ and $\phi_k(y) > 0$ for some $k \in \N$.
For each $\epsilon \in (0,1)$, let $\epsilon' = \epsilon\phi_k(y)\delta(y)$.
Since $\Phi(y)$ is coincident with a function continuous at the point $x_\infty$,
 we can find a subset $E \subset X$ with $\mu(E) = 0$ and an even number $i_0 \geq 4$ such that for every $z, z' \in B(x_\infty,r_{2k + i_0 - 2}) \setminus E$, $|\Phi(y)(z) - \Phi(y)(z')| < \epsilon'/3$.
By the combination of condition (4) with the density of $X \setminus E \subset X$, for every $i \geq i_0$, there is a point $z_i \in B(x_\infty,r_{2k + i_0 - 2}) \setminus E$,
 which is sufficiently close to $x_{2k + i - 1}$, such that $|\Phi(y)(x_{2k + i - 1}) - \Phi(y)(z_i)| < \epsilon'/3$.
Hence for any odd number $i \geq i_0$,
\begin{align*}
 |f_i^k(y) - \delta(y)| &= |f_i^k(y) - f_{i_0}^k(y)|\\
 &= |(\psi_{2k + i}(x_{2k + i - 1})f_i^k(y) + (1 - \psi_{2k + i}(x_{2k + i - 1}))f_{i + 1}^k(y))\\
 &\ \ \ \ \ \ \ \ \ \ \ \ \ \ \ \ - (\psi_{2k + i_0}(x_{2k + i_0 - 1})f_{i_0}^k(y) + (1 - \psi_{2k + i_0}(x_{2k + i_0 - 1}))f_{i_0 + 1}^k(y))|\\
 &= |\Phi(y)(x_{2k + i - 1}) - \Phi(y)(x_{2k + i_0 - 1})|\\
 &\leq |\Phi(y)(x_{2k + i - 1}) - \Phi(y)(z_i)| + |\Phi(y)(z_i) - \Phi(y)(z_{i_0})|\\
 &\ \ \ \ \ \ \ \ \ \ \ \ \ \ \ \ \ \ \ \ \ \ \ \ \ \ \ \ \ \ \ \ \ \ \ \ \ \ \ \ + |\Phi(y)(z_{i_0}) - \Phi(y)(x_{2k + i_0 - 1})|\\
 &< \epsilon'.
\end{align*}
It follows that for each $j \geq (i_0-2)/2$,
\begin{align*}
 g(y)(j) &= (f_{2j+3}^k(y)/\delta(y) - (1 - \phi_k(y))g(y)(j+1))/\phi_k(y)\\
 &> (f_{2j+3}^k(y)/\delta(y) - (1 - \phi_k(y)))/\phi_k(y)\\
 &> ((\delta(y) - \epsilon')/\delta(y) - (1-\phi_k(y)))/\phi_k(y)\\
 &= ((\delta(y) - \epsilon\phi_k(y)\delta(y))/\delta(y) - (1 - \phi_k(y)))/\phi_k(y) = 1 - \epsilon,
\end{align*}
which implies that $g(y) \in \cs_1$.

Finally, we shall verify that $\Phi$ is injective.
Let any $y_1, y_2 \in Y$ with $\Phi(y_1) = \Phi(y_2)$.
Remark that there is $E \subset X$ with $\mu(E) = 0$ such that for each point $x \in X \setminus E$, $\Phi(y_1)(x) = \Phi(y_2)(x)$.
By condition (4) and the density of $X \setminus (\{x_\infty\} \cup E)$ in $X \setminus \{x_\infty\}$, we can see that $\Phi(y_1)|_{X \setminus \{x_\infty\}} = \Phi(y_2)|_{X \setminus \{x_\infty\}}$.
For some $k_i \in \N$, $i = 1, 2$, the point $y_i \in Y_{k_i}$.
Letting $k = \max\{k_i \mid i = 1, 2\}$, we have
\begin{align*}
 \Phi(y_i)(x_{2k + 3}) &= \Phi_{k_i}(y_i)(x_{2k + 3}) = \psi_{2k + 3}(x_{2k + 3})f_{2(k - k_i + 1) + 1}^{k_i}(y_i) + (1 - \psi_{2k + 3}(x_{2k + 3}))f_{2(k - k_i + 2)}^{k_i}(y_i)\\
 &= f_{2(k - k_i + 2)}^{k_i}(y_i) = \delta(y_i),
\end{align*}
 so $\delta(y_1) = \delta(y_2)$.
Thus the both points $y_1$ and $y_2$ are contained in $Y_k$ and
 $$\phi_k(y_1) = 2 - 2^k\delta(y_1) = 2 - 2^k\delta(y_2) = \phi_k(y_2).$$
Furthermore, for every $i \in \N$, we get
 $$f_{i + 1}^k(y_1) = \Phi_k(y_1)(x_{2k + i}) = \Phi(y_1)(x_{2k + i}) = \Phi(y_2)(x_{2k + i}) = \Phi_k(y_2)(x_{2k + i}) = f_{i + 1}^k(y_2),$$
 which means that $f_{j}^k(y_1) = f_{j}^k(y_2)$ for each $j \geq 2$.
When $\phi_k(y_1) = 1$, for all $j \in \N$,
 $$g(y_1)(j) = f_{2j+3}^k(y_1)/\delta(y_1) = f_{2j+3}^k(y_2)/\delta(y_2) = g(y_2)(j).$$
When $\phi_k(y_1) \neq 1$, we see that
 $$g(y_1)(1) = f_3^k(y_1)/((1-\phi_k(y_1))\delta(y_1)) = f_3^k(y_2)/((1-\phi_k(y_2))\delta(y_2)) = g(y_2)(1).$$
Supposing that $g(y_1)(j) = g(y_2)(j)$ for some $j \in \N$, we can obtain
\begin{align*}
 g(y_1)(j + 1) &= (f_{2j + 3}^k(y_1)/\delta(y_1) - \phi_k(y_1)g(y_1)(j))/(1 - \phi_k(y_1))\\
 &= (f_{2j + 3}^k(y_2)/\delta(y_2) - \phi_k(y_2)g(y_2)(j))/(1 - \phi_k(y_2)) = g(y_2)(j + 1).
\end{align*}
By induction, it follows that $g(y_1)(j) = g(y_2)(j)$ for any $j \in \N$,
 that is, $g(y_1) = g(y_2)$.
By virtue of the injectivity of $g$,
 we have $y_1 = y_2$,
 so $\Phi$ is an injection.
The proof is completed.
\end{proof}

\begin{remark}\label{cpt.supp.}
In the above lemma, if there is a compact set $K \subset X$ that contains $B(x_\infty,d(x_1,x_\infty))$,
 the function $\Phi(y)$ has a compact support for each $y \in Y$.
Hence when $\Phi(y)$ is continuous,
 it is uniformly continuous,
 so condition (5) can be rewritten as follows:
\begin{enumerate}
\setcounter{enumi}{4}
 \item $y \in g^{-1}(\cs_1)$ if and only if $\Phi(y) \in C_u(X)$.
\end{enumerate}
\end{remark}

Now we show the following:

\begin{prop}\label{str.univ.pair}
Suppose that for every $E \in \mathcal{M}$ with $\mu(E) = 0$, the complement $X \setminus E$ is dense in $X$,
 and that there are distinct points $x_0, x_\infty \in X_0$ such that for any $\lambda > 0$, $B(x_\infty,\lambda) \in \mathcal{M}$,
 and $x_\infty$ has a compact neighborhood $K \subset X$ with $\mu(K) < \infty$.
If $C_u(X)$ is homotopy dense in $L^p(X)$,
 then the pair $(L^p(X),C_u(X))$ is strongly $(\mathfrak{M}_0,\mathfrak{M}_2)$-universal,
 and hence $C_u(X)$ is strongly $\mathfrak{M}_2$-universal.
\end{prop}

\begin{proof}
The latter half follows from the strong $(\mathfrak{M}_0,\mathfrak{M}_2)$-universality of $(L^p(X),C_u(X))$ and Proposition~\ref{str.univ.}.
We shall show the first half.
Suppose that $(A,D) \in (\mathfrak{M}_0,\mathfrak{M}_2)$, $B$ is a closed set in $A$, and $\Phi : A \to L^p(X)$ is a map such that $\Phi|_B$ is a $Z$-embedding and $(\Phi|_B)^{-1}(C_u(X)) = B \cap D$.
For each $\epsilon > 0$, let us construct a $Z$-embedding\footnote{Recall that every compact set in $\s$ is a $Z$-set, refer to Theorem~1.1.14 and Proposition~1.4.9 of \cite{BRZ}.
Under our assumption in Main Theorem, the space $L^p(X)$ is homeomorphic to $\s$ by Theorem~\ref{Lp},
 and hence the image of any map from $A \in \mathfrak{M}_0$ is a $Z$-set in $L^p(X)$.} $\Psi : A \to L^p(X)$ such that $\|\Psi(a) - \Phi(a)\|_p < \epsilon$ for every $a \in A$, $\Psi|_B = \Phi|_B$ and $\Psi^{-1}(C_u(X)) = D$.
We can assume that $\Phi(B) \cap \Phi(A \setminus B) = \emptyset$ because $\Phi(B)$ is a $Z$-set in $L^p(X)$.
Let $\delta : A \to [0,1)$ be a map defined by
 $$\delta(a) = \min\{\epsilon,\dist(\{\Phi(a)\},\Phi(B))\}/4.$$
As is easily observed,
 $\delta(a) = 0$ if and only if $a \in B$.
Since $C_u(X)$ is homotopy dense in $L^p(X)$,
 there is a homotopy $h : L^p(X) \times [0,1] \to L^p(X)$ such that $h(f,0) = f$, $h(f,t) \in C_u(X)$ for all $f \in L^p(X)$ and $t \in (0,1]$,
 and moreover, $\|h(f,t) - f\|_p \leq t$ for all $f \in L^p(X)$ and $t \in [0,1]$.
Define a map $\phi : A \to L^p(X)$ by setting $\phi(a) = h(\Phi(a),\delta(a))$.
Notice that
 $$\|\phi(a) - \Phi(a)\|_p = \|h(\Phi(a),\delta(a)) - \Phi(a)\|_p \leq \delta(a)$$
 for every $a \in A$,
 and that $\phi(A \setminus B) \subset C_u(X)$.

Take $0 < \lambda \leq d(x_0,x_\infty)/2$ such that $B(x_\infty,\lambda) \subset K$ and $\mu(B(x_\infty,\lambda)) \leq 1$.
According to Lemma~\ref{dig}, we can find maps $\psi : A \setminus B \to L^p(X)$, $\xi : A \setminus B \to (0,\lambda)$ and $\eta : A \setminus B \to (0,\lambda)$ so that for any $a \in A \setminus B$, $\xi(a) \leq \delta(a)$ and
\begin{enumerate}
 \renewcommand{\labelenumi}{(\roman{enumi})}
 \item $\|\phi(a) - \psi(a)\|_p \leq \delta(a)$,
 \item $\psi(a)(B(x_\infty,\xi(a))) = \{0\}$,
 \item $\psi(a)(B(x_0,\eta(a))) = \{0\}$.
\end{enumerate}
Since $\phi(A \setminus B) \subset C_u(X)$,
 we may assume that $\psi(A \setminus B) \subset C_u(X)$,
 see Remark~\ref{target}.
Put
 $$A_k = \{a \in A \mid 2^{-k} \leq \xi(a) \leq 2^{-k+1}\}$$
 for each $k \in \N$,
 so every $A_k$ is compact and $A \setminus B = \bigcup_{k \in \N} A_k$.
It follows from the assumption that $X \setminus \{x_\infty\}$ is dense in $X$.
Choose a point $x_1 \in X \setminus \{x_\infty\}$ such that $d(x_1,x_\infty) < \min\{\xi(a) \mid a \in A_1\}$.
Moreover, we can inductively find $x_m \in X \setminus \{x_\infty\}$ for any $m \geq 2$ so that
 $$d(x_m,x_\infty) < \min\{1/m, d(x_{m - 1},x_\infty), \xi(a) \mid a \in A_m\}.$$
For simplicity, let $r_m = d(x_m,x_\infty)$ for each $m \in \N$.
Then $\{r_m\}$ is strictly converging to $0$ and for any $a \in A_k$ and $k \in \N$, $\psi(a)(B(x_\infty,r_k)) = \{0\}$.
By virtue of the strong $(\mathfrak{M}_0,\mathfrak{M}_2)$-universality of $(\s,\cs_1)$,
 we can obtain an embedding $g : A \to \s$ so that $g^{-1}(\cs_1) = D$.
Applying Lemma~\ref{div.} and Remark~\ref{cpt.supp.}, take an injective map $\psi' : A \setminus B \to L^p(X)$ so that the following conditions hold for every $a \in A \setminus B$.
\begin{enumerate}
 \item $\|\psi'(a)\|_p \leq \xi(a)$.
 \item $\psi'(a)(X \setminus B(x_\infty,r_{2k})) = \{0\}$ if $a \in A_k$, $k \in \N$.
 \item $\psi'(a)(x_m) = \xi(a)$ for all $m \in \{2j + 1, \infty \mid j > k\}$ if $a \in A_k$, $k \in \N$.
 \item $\psi'(a)$ is continuous on $X \setminus \{x_\infty\}$.
 \item $a \in D \setminus B$ if and only if $\psi'(a) \in C_u(X)$.
\end{enumerate}

Let $\psi'' : A \setminus B \to L^p(X)$ be a map defined by $\psi''(a) = \psi(a) + \psi'(a)$.
Since $\psi$ and $\psi'$ are continuous,
 so is $\psi''$.
Due to conditions (i) and (1), for every $a \in A \setminus B$,
\begin{align*}
 \|\phi(a) - \psi''(a)\|_p &= \|\phi(a) - (\psi(a) + \psi'(a))\|_p \leq \|\phi(a) - \psi(a)\|_p + \|\psi'(a)\|_p\\
 &\leq \delta(a) + \xi(a) \leq 2\delta(a).
\end{align*}
By virtue of condition (5), we have that $a \in D \setminus B$ if and only if $\psi''(a) \in C_u(X)$.
To verify that $\psi''$ is an injection, fix any $a_1, a_2 \in A \setminus B$ with $\psi''(a_1) = \psi''(a_2)$,
 where we get some $k_1, k_2 \in \N$ such that $a_1 \in A_{k_1}$ and $a_2 \in A_{k_2}$ respectively.
Let $k = \max\{k_i \mid i = 1, 2\}$.
According to (ii), almost everywhere on $B(x_\infty,r_{2k})$,
 $$\psi'(a_1)(x) = \psi''(a_1)(x) = \psi''(a_2)(x) = \psi'(a_2)(x).$$
Since $\psi'(a_i)$, $i = 1, 2$, is continuous on $B(x_\infty,r_{2k}) \setminus \{x_\infty\}$ by (4),
 and for any $E \in \mathcal{M}$ with $\mu(E) = 0$, $X \setminus E$ is dense in $X$ by the assumption,
 we can see that $\psi'(a_1)(x) = \psi'(a_2)(x)$ for every $x \in B(x_\infty,r_{2k}) \setminus \{x_\infty\}$,
 and therefore especially,
 $$\xi(a_1) = \psi'(a_1)(x_{2k + 3}) = \psi'(a_2)(x_{2k + 3}) = \xi(a_2).$$
Thus $a_1, a_2 \in A_k$.
On the other hand, since $\psi'(a_i)(X \setminus B(x_\infty,r_{2k})) = \{0\}$, $i = 1, 2$, due to condition (2),
 we have that $\psi'(a_1)$ and $\psi'(a_2)$ are coincident almost everywhere on $X$.
It follows from the injectivity of $\psi'$ that $a_1 = a_2$.
Consequently, the map $\psi''$ is an injection.

The map $\psi''$ can be extended to the desired map $\Psi : A \to L^p(X)$ by $\Psi|_B = \Phi|_B$ because for each $a \in A \setminus B$,
\begin{align*}
 \|\Phi(a) - \psi''(a)\|_p &\leq \|\Phi(a) - \phi(a)\|_p + \|\phi(a) - \psi''(a)\|_p \leq 3\delta(a)\\
 &= 3\min\{\epsilon,\dist(\{\Phi(a)\},\Phi(B))\}/4 < \dist(\{\Phi(a)\},\Phi(B)).
\end{align*}
Observe that $\|\Phi(a) - \Psi(a)\|_p < \epsilon$ for any $a \in A$ and that
 $$\Psi(A \setminus B) = \psi''(A \setminus B) \subset L^p(X) \setminus \Phi(B) = L^p(X) \setminus \Psi(B).$$
We see $\Psi^{-1}(C_u(X)) = D$ due to that $a \in D \setminus B$ if and only if $\psi''(a) \in C_u(X)$ and the assumption that $(\Phi|_B)^{-1}(C_u(X)) = B \cap D$.
It remains to prove that $\Psi$ is a $Z$-embedding.
Since $\Psi|_B = \Phi|_B$ is a $Z$-embedding and $\Psi|_{A \setminus B} = \psi''$ is an injective map,
 $\Psi$ is an embedding.
For every $a \in A \setminus B$, $\psi'(a)(X \setminus B(x_\infty,\lambda)) = \{0\}$ by (2),
 and $B(x_0,\eta(a)) \subset X \setminus B(x_\infty,\lambda)$ by the definition.
It follows from (iii) that for any $x \in B(x_0,\eta(a))$,
 $$\Psi(a)(x) = \psi''(a)(x) = \psi(a)(x) + \psi'(a)(x) = 0.$$
According to Lemma~\ref{raise}, the image $\Psi(A) = \Psi(A \setminus B) \cup \Psi(B)$ is a $Z$-set in $L^p(X)$.
We conclude that $\Psi$ is a $Z$-embedding.
\end{proof}

\section{Proof of Main Theorem}

Now we shall prove Main Theorem.
In the case that a measure space $X$ is regular Borel,
 and that there exists a countable family $\{U_n\}$ of open sets in $X$ such that $X = \bigcup_{n \in \N} U_n$, $U_n \subset U_{n + 1}$ and $\mu(U_n) < \infty$ for all $n \in \N$,
 each function of $L^p(X)$ can be approximated by a bounded map that vanishes outside $U_n$ for some $n \in \N$.
Let
 $$C_c(X) = \{f \in L^p(X) \mid f \text{ is a continuous function with a compact support}\}.$$
The space $C_c(X)$ is a convex subset of $L^p(X)$ and $C_c(X) \subset C_u(X)$.
We can show the following proposition.

\begin{prop}\label{homot.dense}
Let $X$ a separable locally compact regular Borel metric measure space.
Suppose that for each compact subset $K \subset X$, $\mu(K) < \infty$.
Then the space $C_c(X)$ is homotopy dense in $L^p(X)$,
 and hence so is $C_u(X)$.
\end{prop}

\begin{proof}
Since $X$ is separable and locally compact,
 we can find a countable family $\{U_n\}$ of open subsets so that $X = \bigcup_{n \in \N} U_n$, $\cl_X{U_n} \subset U_{n + 1}$ and $\cl_X{U_n}$ is compact for any $n \in \N$.
By the assumption, each $\mu(U_n) \leq \mu(\cl_X{U_n}) < \infty$.
Therefore $C_c(X)$ is a dense convex subset in the normed linear space $L^p(X)$.
It follows from the combination of Theorem~6.8.9 with Corollary~6.8.5 of \cite{Sakaik11} that $C_c(X)$ is homotopy dense in $L^p(X)$.
\end{proof}

As is easily observed, if for every $E \in \mathcal{M}$ with $\mu(E) = 0$, the complement $X \setminus E$ is dense in $X$,
 then for any non-empty open subset $U \in \mathcal{M}$, $\mu(U) > 0$.
Moreover, when $\mathcal{M}$ contains the open sets in $X$,
 the converse is valid.
Indeed, supposing that there is a subset $E \in \mathcal{M}$ such that $\mu(E) = 0$ and $X \setminus E$ is not dense in $X$,
 we can obtain an non-empty open subset $U \subset X$,
 that is contained in $E$.
Then
 $$0 < \mu(U) \leq \mu(E) = 0,$$
 which is a contradiction.

\begin{proof}[Proof of Main Theorem]
Remark that for each $E \subset X$ with $\mu(E) = 0$, $X \setminus E$ is dense in $X$.
By Proposition~\ref{M2}, we have $C_u(X) \in \mathfrak{M}_2 \subset (\mathfrak{M}_2)_\sigma$.
Due to Propositions~\ref{homot.dense}, $C_u(X)$ is homotopy dense in $L^p(X)$.
By Proposition~\ref{Zsigma}, there exists a $Z_\sigma$-set in $L^p(X)$ that contains $C_u(X)$.
Since $X \setminus X_0$ is not dense in $X$,
 $X_0$ is uncountable,
 so we can choose distinct points $x_0, x_\infty \in X_0$.
Because $X$ is locally compact and any compact subset is of finite measure,
 the point $x_\infty$ has a compact neighborhood with a finite measure.
According to Proposition~\ref{str.univ.pair}, $C_u(X)$ is strongly $\mathfrak{M}_2$-universal.
Hence the space $C_u(X)$ is an $\mathfrak{M}_2$-absorbing set in $L^p(X)$.
Combining this with Theorems~\ref{Lp} and \ref{abs.}, we conclude that $C_u(X)$ is homeomorphic to $\cs_0$.
\end{proof}

\section{Appendix}

In the theory of infinite-dimensional topology, it is important to consider pairs of spaces.
A pair $(M,Y) \in (\mathfrak{A},\mathfrak{C})_\sigma$ if $M$ can be expressed as a countable union of closed subsets $M_n$, $n \in \N$,
 and $(M_n,M_n \cap Y) \in (\mathfrak{A},\mathfrak{C})$.
We say that $(M,Y)$ is an \textit{$(\mathfrak{A},\mathfrak{C})$-absorbing pair} if the following conditions are satisfied.
\begin{enumerate}
 \item $(M,Y)$ is strongly $(\mathfrak{A},\mathfrak{C})$-universal.
 \item $Y$ is contained in some $Z_\sigma$-set $Z$ of $M$ such that $(Z,Y) \in (\mathfrak{A},\mathfrak{C})_\sigma$.
\end{enumerate}
The pairs $(\s,\cs_0)$ and $(\Q,\cs_0)$ are $(\mathfrak{M}_0,\mathfrak{M}_2)$-absorbing.
As a consequence of Theorem~1.7.6 of \cite{BRZ}, we have the following:

\begin{thm}
Let $\mathfrak{A}$ and $\mathfrak{C}$ be classes of metrizable spaces.
Suppose that both $(M,Y)$ and $(M',Y')$ are $(\mathfrak{A},\mathfrak{C})$-absorbing pairs,
 and both $M$ and $M'$ are topological copies of $\s$ or $\Q$.
Then $(M,Y)$ is homeomorphic to $(M',Y')$.
\end{thm}

The following question naturally arises.

\begin{prm}
Is the pair $(L^p(X),C_u(X))$ homeomorphic to $(\s,\cs_0)$?
\end{prm}

In the paper \cite{YZ}, some continuous function space $C$ endowed with the hypograph topology admits a compactification $\overline{C}$ consisting of upper semi-continuous set valued functions such that the pair $(\overline{C},C)$ is homeomorphic to $(\Q,\cs_0)$.
Let us ask the following:

\begin{prm}
Does the space $C_u(X)$ have a ``natural'' compactification $\overline{C_u(X)}$ such that $(\overline{C_u(X)},C_u(X))$ is homeomorphic to $(\Q,\cs_0)$?
\end{prm}

\end{document}